\documentclass[11pt,a4paper,oneside]{article}

\usepackage{geometry}
\usepackage{hyperref} 

\usepackage{amssymb,amsmath}
\usepackage[amsthm,thmmarks]{ntheorem}
\usepackage{bbm}
\usepackage{enumitem}
\setenumerate{label={\normalfont(\alph*)},topsep=4pt,itemsep=0pt} 

\usepackage{float}
\usepackage{wrapfig}

\usepackage[T2A,T1]{fontenc}
\usepackage[utf8]{inputenc}
\usepackage[russian,UKenglish]{babel}
\newcommand{\SORTNOOPCYR}[1]{} 
\newcommand{\SortNoop}[1]{}

\usepackage{filecontents}
\usepackage{bibentry}
\nobibliography*

\usepackage{graphicx}
\usepackage{tikz}
\usetikzlibrary{matrix,arrows,decorations.pathmorphing}

\usepackage{wrapfig}
\usepackage{subfig}
\usetikzlibrary{calc}
\usepackage{relsize}
\usepackage{xfrac}
\usepackage{color}

\usepackage{numname} 

\usepackage{mathtools}  

\theoremstyle{change}
\theoremseparator{}
\newtheorem{theorem}{Theorem}[section]
\newtheorem{lemma}[theorem]{Lemma}

\newtheorem{corollary}[theorem]{Corollary}

\theoremstyle{change}
\theorembodyfont{\upshape}
\theoremseparator{}
\newtheorem{obs}[theorem]{\!\!}

\theoremstyle{change}
\theorembodyfont{\upshape}
\newtheorem{definition}[theorem]{Definition}

\newtheorem{remark}[theorem]{Remark}
\newtheorem{example}[theorem]{Example}

\theoremstyle{plain} 

\theoremstyle{margin}

\theoremstyle{definition}
\theorembodyfont{\color[rgb]{0,0.7,0}}

\usepackage[thmmarks]{ntheorem}
\theoremstyle{change}
\theorembodyfont{\upshape}
\theoremsymbol{\ensuremath{\clubsuit}}
\theoremseparator{}
\theorembodyfont{\color[rgb]{0.6,0,0.4}}

\theorembodyfont{\itshape\bfseries}
\makeatletter
\g@addto@macro\th@remark{\thm@headpunct{}}
\makeatother

\renewcommand{\epsilon}{\varepsilon}

\newcommand{\1}{\mathbbm{1}}

\newcommand{\B}{\mathbb{B}}
\newcommand{\R}{\mathbb{R}}

\newcommand{\N}{\mathbb{N}}

\newcommand{\cA}{\mathcal{A}}
\newcommand{\cB}{\mathcal{B}}

\newcommand{\cL}{\mathcal{L}}

\newcommand{\cN}{\mathcal{N}}

\newcommand{\cP}{\mathcal{P}}

\newcommand{\fB}{\mathfrak{B}} 	\newcommand{\fb}{\mathfrak{b}}

\newcommand{\fU}{\mathfrak{U}}\newcommand{\fu}{\mathfrak{u}}

\newcommand{\sumn}{\sum_{n\in \N}}

\newcommand{\limn}{\lim_{n\rightarrow \infty}}

\DeclareMathOperator{\D}{\, \mathrm{d}\!}

\usepackage{appendix}
\usepackage{chngcntr}
\usepackage{etoolbox}

\AtBeginEnvironment{subappendices}{%
\chapter*{Appendix}
\addcontentsline{toc}{chapter}{Appendices}
\counterwithin{figure}{section}
\counterwithin{table}{section}
}

\usetikzlibrary{matrix,arrows,decorations.pathmorphing}

\usepackage{numname} 

\renewcommand{\v}{v} 

\newcommand{\unifrightarrow}{\xrightarrow{u}}

\newcommand{\todolater}{\todo[color=yellow!50]}

\usepackage[disable]{todonotes}

\begin{document}

\renewcommand{\thefootnote}{\Roman{footnote}}

\title{Bochner integrals in ordered vector spaces}


\author{
\renewcommand{\thefootnote}{\Roman{footnote}}
A.C.M. van Rooij
\footnotemark[1]
\\
\renewcommand{\thefootnote}{\Roman{footnote}}
W.B. van Zuijlen
\footnotemark[2]
}

\footnotetext[1]{
Radboud University Nijmegen, Department of Mathematics, P.O.\ Box
9010, 6500 GL Nijmegen, the Netherlands.
}
\footnotetext[2]{
Leiden University, Mathematical Institute, P.O.\ Box 9512, 2300 RA, Leiden, the
Netherlands.
}

\maketitle

\renewcommand{\thefootnote}{\arabic{footnote}} 

\begin{abstract}
We present a natural way to cover an Archimedean directed ordered vector space $E$ by Banach spaces and extend the notion of Bochner integrability to functions with values in $E$. 
The resulting set of integrable functions is an Archimedean directed ordered vector space and the integral is an order preserving map. 

\todo{dingen weggelaten}

\bigskip\noindent
{\it Mathematics Subject Classification (2010).} 28B05, 28B15. \\
{\it Keywords:} Bochner integral, ordered vector space, ordered Banach space, closed cone, generating cone.
\end{abstract}


\section{Introduction}

We extend the notion of Bochner integrability to functions with values in a 
vector space $E$ that may not itself be a Banach space but is the union of a collection $\B$ of Banach spaces. 

The idea is the following. 
We call a function $f$, defined on a measure space $X$ and with values in $E$, ``integrable'' if for some $D$ in $\B$ all values of $f$ lie in $D$ and $f$ is Bochner integrable as a function $X \rightarrow D$. 
Of course, one wants a certain consistency: the ``integral'' of such an $f$ should be independent of the choice of $D$. 

In \cite{Th75}, Thomas 
obtains this consistency by assuming a Hausdorff locally convex topology on $E$, 
entailing many continuous linear functions $E \rightarrow \R$. 
Their restrictions to the Banach spaces that constitute $\B$ enable one to apply Pettis integration, which leads to the desired uniqueness. 

Our approach is different, following a direct-limit-like construction. 
We assume $E$ to be an ordered vector space with some simple regularity properties (Archimedean, directed) and show that $E$ is the union of a certain increasing system $\B$ of Banach spaces with closed, generating positive cones (under the ordering of $E$). 
Uniqueness of the integral follows from properties of such ordered Banach spaces. 
Moreover, the integrable functions form a vector space and the integral is linear and order preserving. 

In Section \ref{section:useful_banach_spaces} we study ordered Banach spaces with closed generating cones. We give certain properties which can be used to give an alternative proof of a classical theorem which states that every order preserving linear map is continuous, and generalise it to order bounded linear maps. 
In Section \ref{section:bochner_integral_on_ordered_Banach_spaces} we study Bochner integrable functions with values in an ordered Banach space with closed generating cone. 
In Section \ref{section:generalised_Bochner} we present the definition of a Banach cover and the definition of the extension of the Bochner integral to functions with values in a vector space that admits a Banach cover. 
In Section \ref{section:useful_covers} we show that an Archimedean ordered vector space possesses a Banach cover consisting of ordered Banach spaces spaces with closed generating cones. 
In Section \ref{section:Bochner_on_covers2} we study integrable functions with values in Archimedean ordered vector spaces. 
In Section \ref{section:comparing_integrals} we compare the integral with integrals considered in \cite{vRvZ}. 
In Section \ref{section:convolution} we present an application to view the convolution as an integral.

\section{Notation}

$\N$ is $\{1,2,\dots\}$. 
We write ``for all $n$'' instead of ``for all $n\in\N$''. 
To avoid confusion:
\begin{itemize}
\item An ``order'' is a ``partial order''. 
\item 
We call an ordered vector space \emph{Archimedean} (see Peressini \cite{Pe67}) if for all $a,b\in E$ the following holds: if $na \le b$ for all $n\in\N$, then $a\le 0$. 
(In some places, e.g., Birkhoff \cite{Bi67}, such spaces are said to be `integrally closed'.)
\end{itemize}
As is common in literature, 
our notations do not 
distinguish between a function on a measure space and the class of that function.

\section{Ordered Banach spaces with closed generating cones}
\label{section:useful_banach_spaces}

In this section we describe properties of ordered Banach spaces with closed generating cones. Using these properties we prove in Theorem \ref{theorem:order_bounded_linear_between_useful_implies_continuous} that an order bounded map between ordered Banach spaces with closed generating cones is continuous. 

\begin{definition}
A \emph{normed ordered vector space} is a normed vector space with an order that makes it an ordered vector space. 
An \emph{ordered Banach space} is a Banach space that is a normed ordered vector space. 
\end{definition}

A priori there is no connection between the ordering and the norm of a normed ordered vector space. One reasonable and useful connection is the assumption that the (positive) cone be closed. 

\begin{theorem}
\label{theorem:ordering_determined_by_dual_cone}
Let $E$ be a normed ordered vector space. 
$E^+$ is closed if and only if 
\begin{align}
\label{eqn:positive_of_dual_determines_ordering}
x \le y  \iff \alpha(x) \le \alpha(y) \mbox{ for all } \alpha \in (E')^+. 
\end{align}
Consequently, whenever $E^+$ is closed then $(E')^+$ separates the points of $E$ and $E$ is Archimedean. 
\end{theorem}
\begin{proof}
Since $E^+$ is convex, 
$E^+$ is closed if and only if it is weakly closed (i.e., $\sigma(E,E')$-closed); see \cite[Theorem V.1.4]{Co07}. 
The rest follows by \cite[Theorem 2.13(3\&4)]{AlTo07}.
\end{proof}

The following theorem is due to And\^o \cite{An62}. See also \cite[Corollary 2.12]{AlTo07}. 

\begin{theorem}
\label{theorem:ando}
Let $D$ be an ordered Banach space with a closed generating\footnote{$D^+$ is generating if $D=D^+ - D^+$.} cone $D^+$. 
There exists a $C>0$ such that 
\begin{align}
\label{eqn:C_bound_norm}
C\|x\| &\ge \inf \{ \|a\| : a\in D^+, - a \le x \le a\} \qquad (x\in E). 
\end{align}
\end{theorem}

\begin{definition}
\label{def:norm_C-absolutely_dominating}
Let $D$ be a directed\footnote{$D$ is directed if $D=D^+- D^+$, i.e., if $D^+$ is generating.}
 ordered Banach space. If $C>0$ is such that \eqref{eqn:C_bound_norm} holds, then we say that the norm $\|\cdot\|$ is 
\emph{C-absolutely dominating}.\footnote{Batty and Robinson \cite{BaRo84} call the cone $D^+$ \emph{approximately $C$-absolutely dominating}, and, Messerschmidt \cite{Me15} calls $D$ \emph{approximately $C$-absolutely conormal}, if the norm on $D$ is $C$-absolutely dominating.} 
We say that a norm $\|\cdot\|$ is \emph{absolutely dominating} 
if it is $C$-absolutely dominating for some $C>0$. 

On a Banach lattice the norm is $1$-absolutely dominating. 
Actually for Banach lattices there is equality in \eqref{eqn:C_bound_norm}. 
\end{definition}

\begin{obs}
\label{obs:facts_about_absolutely_dominating}
We refer the reader to Appendix \ref{section:appendix_useful} for the following facts:
If $\|\cdot\|$ is $C$-absolutely dominating on a directed ordered Banach space $D$, then $C\ge 1$. 
Whenever there exists a absolutely dominating norm, then for all $\epsilon>0$ there exists an equivalent $(1+\epsilon)$-absolutely dominating norm. 
All norms on a directed ordered vector space $D$ that make $D$ complete and $D^+$ closed are equivalent (see \ref{obs:norms_on_ordered_Banach_space_with_closed_generating_cone_are_equivalent}). 
\end{obs}

\begin{obs}
\label{obs:union_of_intervals_is_nbh_0_iff_C_abs_dom_norm}
Let $D$ be a directed ordered Banach space. 
Then $\|\cdot\|$ is absolutely dominating 
 if and only if 
the (convex) set 
\begin{align}
\bigcup_{a\in D^+, \|a\|\le 1} [-a,a]
\end{align}
 is a neighbourhood of $0$.  
\end{obs}

\begin{definition}(See \cite[\S 16]{LuZa71})
Let $E$ be an ordered vector space. 
We say that a sequence $(x_n)_{n\in\N}$ in $E$ 
\emph{converges uniformly}
to an element $x\in E$ (notation: $x_n \unifrightarrow x$)
whenever there exist $a\in E^+$, $\epsilon_n \in (0,\infty)$ with $\epsilon_n\rightarrow 0$ and 
\begin{align}
-\epsilon_n a \le x_n -x \le \epsilon_n a \qquad (n\in\N). 
\end{align}
Note that one may replace ``$\epsilon_n \rightarrow 0$'' by ``$\epsilon_n \downarrow 0$''. 
If $E$ is Archimedean then the $x$ as above is unique. 
We will only consider such convergence in Archimedean spaces. 
We say that a sequence $(x_n)_{n\in\N}$ in $E$ is a
\emph{uniformly Cauchy sequence} if there exists an $a\in E^+$ such that for all $\epsilon>0$ there exists an $N$ such that $-\epsilon a \le x_n - x_m \le \epsilon a$ for all $n,m\ge N$. 
$E$ is called \emph{uniformly complete} whenever it is Archimedean and all  uniformly Cauchy sequences converge uniformly. 
\end{definition}

\begin{lemma}
\label{lemma:order_bounded_linear_implies_continuous_wrt_ru}
Let $D_1,D_2$ be ordered vector spaces and $T: D_1 \rightarrow D_2$. 
If $T$ is linear and order bounded, then $T$ preserves uniform convergence.
\end{lemma}
\begin{proof}
Suppose $x_n\in D_1, \epsilon_n \in (0,1)$ and $a\in D_1^+$ are such that $-\epsilon_n a \le x_n \le \epsilon_n a$ and $\epsilon_n \downarrow 0$. 
Let $b\in D_2^+$ be such that $T[-  a , a] \subset [-b,b]$.
Then $\epsilon_n^{-1}Tx_n \in [-b,b]$, i.e., $-\epsilon_n b \le Tx_n \le \epsilon_n b$ for all $n$. 
\end{proof}


\begin{theorem}
\label{theorem:equivalences_useful}
Let $D$ be an ordered Banach space. 
Consider the following conditions. 
\begin{enumerate}[label=\normalfont{(\roman*)}]
\item
\label{item:C_norm_bound}
$\|\cdot\|$ is absolutely dominating. 
\item If $a_1,a_2,\dots \in D^+$, $\sumn \|a_n\| <\infty$, then there exist $a\in D^+$, $\epsilon_n \in (0,\infty)$ with $\epsilon_n \rightarrow 0$ such that $a_n \le \epsilon_n a$ for all $n$.\footnote{For normed Riesz spaces property \ref{item:condition_bounded} is equivalent to what is called the weak Riesz-Fischer condition (see Zaanen \cite[Ch. 14, \S 101]{Za83}).}
\label{item:condition_bounded}
\item If 
$x_1,x_2,\dots \in D, \sumn \|x_n\|<\infty $, then $x_n \unifrightarrow 0. $
\label{item:condition_ru_convergence_summable_seq}
\end{enumerate}
If $D^+$ is closed, then $D$ satisfies \ref{item:condition_bounded}. \\
Suppose $D$ is Archimedean and directed. 
Then the following are equivalent.
\begin{enumerate}
\item $D^+$ is closed. 
\label{item:useful}
\item $D$ satisfies \ref{item:C_norm_bound} and \ref{item:condition_bounded}. 
\label{item:archi_and_positive_ru_and_bound_useful}
\item $D$ satisfies \ref{item:condition_ru_convergence_summable_seq}. 
\label{item:archi_and_ru_condition}
\end{enumerate}
\end{theorem}
\begin{proof}
Suppose $D^+$ is closed. 
Let $a_1,a_2,\dots \in D^+$, $\sumn \|a_n\| <\infty$. 
Choose $\epsilon_n \in (0,\infty)$, $\epsilon_n \rightarrow 0$ such that $\sumn \epsilon_n^{-1} \|a_n\| <\infty$. 
By norm completeness, $a:= \sumn \epsilon_n^{-1} a_n$ exists. 
Because $D^+$ is closed $a \ge \epsilon_n^{-1} a_n$ for all $n$. \\
\ref{item:useful}$\Longrightarrow$\ref{item:archi_and_positive_ru_and_bound_useful}. 
\ref{item:condition_bounded} is implied by the above argument.
By 
Theorem \ref{theorem:ando} we have \ref{item:C_norm_bound}.  \\
\ref{item:archi_and_positive_ru_and_bound_useful}$\Longrightarrow$\ref{item:archi_and_ru_condition}. 
Let $x_1,x_2,\dots\in D$, $\sumn \|x_n\|<\infty$. 
Using \ref{item:C_norm_bound} let $a_1,a_2,\dots \in D^+$ with $\sumn \|a_n\|<\infty$ be such that $-a_n \le x_n \le a_n$. 
By \ref{item:condition_bounded} it then follows that $x_n \unifrightarrow 0$. \\
\ref{item:archi_and_ru_condition}$\Longrightarrow$ \ref{item:useful}. 
Take $b$ in the closure of $D^+$. 
For $n\in\N$, choose $x_n\in D^+$, $\sumn  \|x_n - b\|<\infty$. 
By \ref{item:condition_ru_convergence_summable_seq} there exist $a\in D^+$, $\epsilon_n \in (0,\infty)$ with $\epsilon_n \rightarrow 0$ and $-\epsilon_n a \le x_n - b \le \epsilon_n a$, so that $b \ge x_n - \epsilon_n a \ge - \epsilon_n a$. 
Then $b\ge 0$ because $D$ is Archimedean. 
\end{proof}

\begin{lemma}
\label{lemma:cone_closed_implies_norm_and_ru_limit_coincide}
Let $D$ be an ordered  Banach space for which $D^+$ is closed. 
Then 
\begin{align}
x_1,x_2,\dots \in D, \
 a\in D, \ 
 b\in D, \ 
 \|x_n - a\| \rightarrow 0, \
 x_n \unifrightarrow b 
\quad \Longrightarrow \quad a=b. 
\end{align}
\end{lemma}
\begin{proof}
Assume $b=0$. 
There exist $c\in D^+$, $\epsilon_n \in (0,\infty)$ with $\epsilon_n \downarrow 0$ and $-\epsilon_n c \le x_n \le \epsilon_n c$. Then $-\epsilon_N c \le x_n \le \epsilon_N c$ whenever $n\ge N$. Since $D^+$ is closed, $-\epsilon_N c \le a \le \epsilon_N c$ for all $N$, and so, as $D$ is Archimedean by Theorem \ref{theorem:ordering_determined_by_dual_cone}, $a=0$. 
\end{proof}

With this we can easily prove the following theorem. 
In a recently published book by  Aliprantis and Tourky \cite{AlTo07}\footnote{For complete metrisable ordered vector spaces.}
but also in older Russian papers by Wulich \cite{Vu77}
one can find the proof that an order preserving linear map is continuous 
(see \cite[Theorem 2.32]{AlTo07} or combine \cite[Theorem III.2.2]{Vu77} (which states the result for $D_2=\R$) and \cite[Theorem VI.2.1]{Vu78}). 
Theorem \ref{theorem:order_bounded_linear_between_useful_implies_continuous} is more general in the sense that it states that linear order bounded maps are continuous. 

\begin{theorem}
\label{theorem:order_bounded_linear_between_useful_implies_continuous}
Let $D_1,D_2$ be ordered Banach spaces. 
Suppose $D_1^+$ is closed and generating and $D_2^+$ is closed. 
Let $T: D_1  \rightarrow D_2$ be linear and order bounded. 
Then $T$ is continuous. 
Consequently, if $T$ is an order isomorphism then it is a homeomorphism. 
\end{theorem}
\begin{proof}
Let $x_1,x_2,\dots \in D_1$, $x_n \rightarrow 0$ and suppose $T x_n \rightarrow c$ for some $c\in D_2$. 
If from this we can prove $c=0$, then by the Closed Graph Theorem $T$ will be continuous. 
We may assume $\sumn \|x_n\| <\infty$. 
Then $x_n \unifrightarrow 0$ in $D_1$ (Theorem \ref{theorem:equivalences_useful}), so $T x_n  \unifrightarrow 0$ in  $D_2$ (Lemma \ref{lemma:order_bounded_linear_implies_continuous_wrt_ru}). Hence, $c=0$ according to Lemma \ref{lemma:cone_closed_implies_norm_and_ru_limit_coincide}. 
\end{proof}

We present a consequence as has been done for order preserving linear maps in \cite[Theorem VI.2.2]{Vu78}. 

\begin{corollary} 
Let $D_1$ be an ordered Banach space with closed generating cone. 
Let $D_2$ be an ordered normed vector space with a normal cone. 
Let $T: D_1  \rightarrow D_2$ be linear and order bounded. 
Then $T$ is continuous. 
\end{corollary}
\begin{proof}
As $D_2^+$ is normal, we have $D_2^{\sim} = D_2'$ (see \cite[Corollary 2.27]{AlTo07}). 
Let $\Delta$ be the unit ball in $D_1$. 
Then for all $\phi \in D_2'$, the map $\phi \circ T$ is an order bounded functional, whence continuous by Theorem \ref{theorem:order_bounded_linear_between_useful_implies_continuous}. 
Thus $\phi \circ T( \Delta)$ is bounded for all $\phi \in D_2'$. 
By an application of the Principle of Uniform Boundedness (see \cite[Corollary III.4.3]{Co07}), $T(\Delta)$ is norm bounded. 
\end{proof}


\section{The Bochner integral on ordered Banach spaces}
\label{section:bochner_integral_on_ordered_Banach_spaces}

\textbf{In this section $(X,\cA,\mu)$ is a complete $\sigma$-finite measure space, with $\mu \ne 0$. 
For more assumptions see \ref{obs:assumptions_space_useful} and \ref{obs:assumption_D_ordered_Banach space}. }

We define the integral of simple functions in Definition \ref{obs:simple_functions_and_their_integrals} and recall the definition and facts on Bochner integrability in Definition \ref{definition:Bochner_integral} and \ref{obs:facts_on_Bochner}. 
After that, we consider an ordered Banach space $D$. 
In Theorems \ref{theorem:Bochner_functions_Banach_lattice} and \ref{theorem:closed_cone_implies_order_preservingness_integral} we describe the order structure of the space $\fB_D$ of Bochner integrable functions. 
In \ref{obs:comming_results_on_relation_Banach_space_and_Bochner_space} we summarise the  results of \ref{obs:D_closed_cone_then_Bochner_int_functions_too} -- \ref{theorem:closed_and_generating_for_bochner_and_space}, in which we compare closedness and generatingness of the positive cones of $D$ and $\fB_D$. 

\begin{definition}
\label{obs:simple_functions_and_their_integrals}
Let $E$ be a vector space.
We say that a function $f: X\rightarrow E$ is \emph{simple} if there exist $N\in\N$, $a_1,\dots,a_N\in E$, $A_1,\dots, A_N\in \cA$ with $\mu(A_1),\dots,\mu(A_N)<\infty$ for which 
\begin{align}
\label{eqn:representation_simple_functions}
f = \sum_{n=1}^N a_n \1_{A_n}.
\end{align}
The simple functions form a linear subspace $S$ of $E^X$, which is a Riesz subspace of $E^X$ in case $E$ is a Riesz space. 
We define $\varphi : S\rightarrow E$ by 
\begin{align}
\label{eqn:integral_for_simple_functions}
\varphi(f) = \sum_{n=1}^N \mu(A_n) a_n,
\end{align}
where $f,N,A_n,a_n$ are as in \eqref{eqn:representation_simple_functions}. 
$\varphi(f)$ is called the 
\emph{integral} of $f$. 
We write $S_\R$ for the linear space of simple functions $X\rightarrow \R$. 
\end{definition}

Definition \ref{definition:Bochner_integral} and the facts in \ref{obs:facts_on_Bochner} can be found in Chapter III in the book by E. Hille and R.S. Phillips, \cite{HiPh57}. 

\begin{definition}
\label{definition:Bochner_integral}
Let $(D,\|\cdot\|)$ be a Banach space. 
A function $f: X \rightarrow D$ is called 
\emph{Bochner integrable} 
whenever there exists a sequence of simple functions $(s_n)_{n\in\N}$ such that 
\begin{align}
\label{eqn:def_bochner}
\int \|f(x) -s_n(x)\| \D \mu(x) \rightarrow 0 . 
\end{align}
Then the sequence $(\varphi(s_n ) )_{n\in\N}$ converges.
Its limit is independent of the choice of the sequence $(s_n)_{n\in\N}$
 and is called the 
\emph{Bochner integral}
of $f$. 
\end{definition}

\begin{obs}
\label{obs:assumptions_space_useful}
\textbf{For the rest of this section, $D$ is a Banach space (with norm $\|\cdot\|$), and, we write $\fB$ (or $\fB_D$) for the Banach space of classes of Bochner integrable functions $X \rightarrow D$, with norm $\|\cdot\|_\fB$ (see \ref{obs:facts_on_Bochner}\ref{item:Bochner_Banach}). 
We write $\fb$ (or $\fb_D$) for the Bochner integral on $\fB$. 
}
\end{obs}

\begin{obs}
\label{obs:facts_on_Bochner}
Some facts on the Bochner integrable functions: 
\begin{enumerate}
\item \cite[Theorem 3.7.4]{HiPh57} \& \cite[Proposition 2.15]{Ry02}
\label{item:Bochner_integrable_when_measurable_and_norm_integrable}
If $f: X \rightarrow D$ is Borel measurable, $f(X)$ is separable and $\int \|f\| \D \mu <\infty$, then $f$ is Bochner integrable. 
\item \cite[Theorem 3.7.8]{HiPh57} 
\label{item:Bochner_Banach}
$\fB$ is a Banach space under the norm $\|\cdot\|_\fB : \fB \rightarrow [0,\infty) $ given by 
\begin{align}
\label{eqn:norm_Bochner}
\|f\|_\fB =  \int \|f\| \D \mu. 
\end{align}
\item \cite[Theorem 3.7.12]{HiPh57} 
\label{item:Bochner_and_closed_map}
Let $E$ be a Banach space and $T: D \rightarrow E$ be linear and continuous. 
If $f\in \fB_D$, then  $T\circ f \in \fB_E$ and 
\begin{align}
T( \fb_D(f)) = \fb_E ( T \circ f ). 
\end{align}
\item \cite[Theorem 3.7.9]{HiPh57}
\label{item:dominated_convergence_Bochner}
Let $f_1,f_2,\dots$ be in $\fB$, $f: X \rightarrow D$ and $h \in \cL^1(\mu)^+$. 
Suppose that $f_n(x) \rightarrow f(x)$ and $\|f_n(x)\|\le h(x)$ for $\mu$-almost all $x\in X$. Then $f \in \fB$ and 
\begin{align}
\fb(f_n) \rightarrow \fb(f). 
\end{align}
\item 
\label{item:ordering_on_Bochner_functions}
If $D$ is an ordered Banach space then 
$\fB$ is an ordered Banach space under the ordering given by 
\begin{align}
f \le g \quad (\mbox{in } \fB) \iff f \le g \quad \mu-\mbox{a.e.}.
\end{align}
\end{enumerate}
\end{obs}


\begin{theorem}
\label{theorem:Bochner_functions_Banach_lattice}
Let $D$ be a Banach lattice. 
Then $\fB_D$ is a Banach lattice and $\fb$ is linear and order preserving. 
\end{theorem}
\begin{proof}
The Bochner integrable functions form a Riesz space because of the inequality $\| |x| - |y| \| \le \|x-y\|$. 
\end{proof}

\begin{theorem}
\label{theorem:closed_cone_implies_order_preservingness_integral}
Let $D$ be an ordered Banach space for which $D^+$ is closed. 
Then $\fb$ is order preserving. 
\end{theorem}
\begin{proof}
Let $f\in \fB$ and $f\ge 0$. 
Then 
\begin{align}
\alpha ( \fb(f) ) = \int \alpha \circ f \D \mu \ge 0 \qquad (\alpha \in (D')^+). 
\end{align}
Whence $ \fb(f) \ge 0 $, by Theorem \ref{theorem:ordering_determined_by_dual_cone}. 
\end{proof}

\begin{obs}
\label{obs:assumption_D_ordered_Banach space}
\textbf{For the rest of this section, $D$ is an ordered Banach space.}
\end{obs}

\begin{obs}
\label{obs:comming_results_on_relation_Banach_space_and_Bochner_space}
The following is a list of results presented in \ref{obs:D_closed_cone_then_Bochner_int_functions_too} -- \ref{theorem:closed_and_generating_for_bochner_and_space}. 
\begin{enumerate}
\item 
	$\fB^+$ is closed if and only if $D^+$ is (\ref{obs:D_closed_cone_then_Bochner_int_functions_too}). 
\item 
	If $\fB$ is directed, then so is $D$ (straightforward, see also Theorem \ref{theorem:same_bound_norm_space_and_bochner}). 
\item 
	Let $C>0$ and $D^+$ be closed and generating. If $\|\cdot\|$ is $C$-absolutely dominating, then so is $\|\cdot\|_\fB$ (Lemma \ref{lemma:norm_bound_for_bochner}). 
\item 
	Let $C>0$. If $\fB$ is directed and $\|\cdot\|_\fB$ is $C$-absolutely dominating, then so is $\|\cdot\|$ (Theorem \ref{theorem:same_bound_norm_space_and_bochner}). 
\item 
	$\fB^+$ is closed and generating if and only if $D^+$ is closed and generating (Theorem \ref{theorem:closed_and_generating_for_bochner_and_space}). 
\item 
	If there exist disjoint $A_1,A_2,\dots$ in $\cA$ with $0< \mu(A) <\infty$ for all $n$: 
	If $\fB^+$ is generating, then so is $D^+$ and $\|\cdot\|$ is absolutely dominating (Corollary \ref{cor:bochner_directed_when_countable_pos_sets_implies_norm_bound}). 
\item 
	If no such $A_1,A_2,\dots$ exist: $\fB^+$ is generating if and only if $D^+$ is generating (\ref{obs:bochner_directed_finite_atoms}).
\end{enumerate}
\end{obs}

\begin{obs}
\label{obs:D_closed_cone_then_Bochner_int_functions_too}
Whenever $f\in \fB$, $f_n \in \fB^+$ with $\int \|f-f_n\|\D \mu \rightarrow 0$, then there exist Bochner integrable $g_n\ge 0$ with $\sumn \int \|f-g_n\|\D \mu <\infty$; this implies $g_n \rightarrow f$ $\mu$-almost everywhere. So whenever $D^+$ is closed this implies $f\ge 0$ $\mu$-almost everywhere. \\
We infer that $\fB^+$ is closed whenever $D^+$ is. \\
On the other hand, if $\fB^+$ is closed then so is $D^+$. 
Indeed, let $A\in \cA$, $0<\mu(A)<\infty$. 
If $a_n\in D^+$ and $a_n \rightarrow a$, then $\|a\1_A - a_n \1_A \|_{\fB} \rightarrow 0$.  Therefore $a\in D^+$. 
\end{obs}

\begin{lemma}
\label{lemma:simple_function_norm_close}
Suppose $D^+$ is generating and $C>0$ is such that $\|\cdot\|$ is $C$-absolutely dominating. 
Let $f: X \rightarrow D$ be simple, let $\epsilon>0$. 
Then there exists a simple $g: X \rightarrow D^+$ with $-g \le f \le g$ and $\int \|g\| \D \mu \le C \int \|f\|\D \mu + \epsilon$. 
\end{lemma}
\begin{proof}
Write $f = \sum_{n=1}^N x_n \1_{A_n}$ with disjoint sets $A_1,\dots, A_N\in \cA$ of finite measure and  $x_1,\dots, x_N\in D$. 
Let $\kappa = \mu(\bigcup_{n=1}^N A_n)$ and assume $\kappa >0$. 
For each $n$, choose $a_n\in D^+$, $-a_n \le x_n \le a_n$, $\|a_n\|\le C\|x_n\| + \frac{\epsilon}{\kappa}$. 
Put $g = \sum_{n=1}^N a_n \1_{A_n}$.
Then $-g \le f \le g$ and $\int \|g\|\D \mu = \sum_{n=1}^N \|a_n\| \mu(A_n) \le \sum_{n=1}^N (C\|x_n\| +\frac{\epsilon}{\kappa}) \mu(A_n) = C\int \|f\|\D \mu + \epsilon$. 
\end{proof}

\begin{lemma}
\label{lemma:norm_bound_for_bochner}
Suppose $D^+$ is closed and generating and $C>0$. 
Then $\fB^+$ is closed and generating. 
If $\|\cdot\|$ is $C$-absolutely dominating, then so is $\|\cdot\|_\fB$. 
\end{lemma}
\begin{proof}
$\fB^+$ is closed by \ref{obs:D_closed_cone_then_Bochner_int_functions_too}.
Assume that $\|\cdot\|$ is $C$-absolutely dominating. 
Let $f: X \rightarrow D$ be Bochner integrable and let $\epsilon>0$. 
We prove there exists a Bochner integrable $g: X \rightarrow D^+$ with $-g \le f \le g$ $\mu$-a.e. and 
$\int \|g\| \D \mu \le C (\int \|f\| \D \mu +\epsilon) $. 
Choose simple functions $s_1,s_2,\dots$ with $\int \|f-s_n\| \D \mu < \epsilon 2^{-n-1}$. 
Then $s_n \rightarrow f$ pointwise outside a $\mu$-null set $Y$. 

Define $f_1 = s_1$, $f_n = s_n - s_{n-1}$ for $n\in \{2,3,\dots\}$. Observe 
\begin{align}
\notag \int \|f_1\|\D \mu &< \int \|f\|\D \mu + \epsilon 2^{-2}, \\
\int \|f_n\|\D \mu &< \epsilon 2^{-n-1} + \epsilon 2^{-n-2} < \epsilon 2^{-n} \qquad (n\in \{2,3,\dots\}). 
\end{align}
For each $n$, choose a simple $g_n : X \rightarrow D^+$ with $-g_n \le f_n \le g_n$ such that 
\begin{align}
\notag \int \|g_1\| \D \mu &< C \left( \int \|f\| \D \mu + \epsilon 2^{-2} \right), \\
\int \|g_n\| \D \mu &< C\epsilon 2^{-n} \qquad (n\in \{2,3,\dots\}). 
\end{align}
As $\sumn \int \|g_n\| \D \mu <\infty$ there is a $\mu$-null set $Z\subset X$ for which 
\begin{align}
\sumn \| g_n(x)\| <\infty \qquad (x\in X \setminus Z). 
\end{align}
Put $X_0 := X \setminus (Y\cup Z)$. Define $g: X \rightarrow D^+$ by 
\begin{align}
g(x) =
\begin{cases}
 \sumn g_n(x) & x\in X_0, \\
 0 & x\notin X_0. 
\end{cases}
\end{align}
Then $g$ is Bochner integrable and $\int \|g\| \D \mu \le \sumn \int \|g_n\| \D \mu \le C(\int \|f\|\D \mu + \epsilon)$. 

Moreover,  for $x\in X_0$  we have $-g(x) \le \sum_{n=1}^N f_n(x) = s_N(x) \le g(x)$
 for all $N$, whereas $s_N(x) \rightarrow f(x)$ since $x\notin Y$. 
From the closedness of $D^+$ it follows that $-g \le f \le g$ on $X_0$. 
\end{proof}

In the following theorems (\ref{theorem:same_bound_norm_space_and_bochner},
\ref{theorem:directed_N_to_D_Bochner_implies_norm_bound_useful}, \ref{cor:bochner_directed_when_countable_pos_sets_implies_norm_bound} and \ref{obs:bochner_directed_finite_atoms}) we derive properties of $D$ from properties of $\fB$. 
In Theorem  \ref{theorem:closed_and_generating_for_bochner_and_space} we show that $D$ has a closed and generating cone if and only if $\fB$ does. 

\begin{theorem}
\label{theorem:same_bound_norm_space_and_bochner}
Assume $\fB$ is directed.
Let $C>0$ and suppose $\|\cdot\|_{\fB}$ is $C$-absolutely dominating.
Then $D$ is directed and $\|\cdot\|$ is $C$-absolutely dominating.
\end{theorem}
\begin{proof}
Let $x\in D$ and $x\ne 0$. 
Let $A\in \cA$ with $0<\mu(A) <\infty$.
Then $x\1_A\in \fB$, $\|x\1_A\|_{\fB} = \mu(A) \|x\| $. 
Let $C'>C$. 
There is a $g\in \fB$ with $-g \le x\1_A \le g$, 
$\int \|g(t)\| \D \mu(t) \le C' \|x\1_A\|_{\fB}$. 
Then 
$\int_A \|g(t)\| \D \mu(t) \le C' \mu(A) \|x\|$, 
so $\mu\{t\in A : \|g(t)\|> C' \|x\|\} < \mu(A)$. 
In particular, there is a $t\in A$ with $\|g(t)\| \le C' \|x\|$ and $-g(t) \le x \le g(t)$. 
\end{proof}

\begin{theorem}
\label{theorem:directed_N_to_D_Bochner_implies_norm_bound_useful}
Let $D$ be an ordered Banach space such that the Bochner integrable functions $\N \rightarrow D$ form a directed space. 
Then $D$ is directed and $\|\cdot\|$ is absolutely dominating. 
\end{theorem}
\begin{proof}
$D$ is directed. In case $\|\cdot\|$ is not absolutely dominating, there exist $x_1,x_2,\dots \in D$ such that for every $n$ 
\begin{align}
\notag &2^n \|x_n\| \le \inf \{ \|a\|: a\in D^+, - a \le x_n \le a\}
\end{align}
and $\|x_n\| = 2^{-n}$. Then $n\mapsto x_n$ is Bochner integrable, so, by our assumption, there exist $a_n \in D^+$ with $-a_n \le x_n \le a_n$ for all $n$ and $ \sumn \|a_n\|<\infty$ which is false.
\end{proof}

\begin{corollary}
\label{cor:bochner_directed_when_countable_pos_sets_implies_norm_bound}
Suppose there exist disjoint $A_1,A_2,\dots$ in $\cA$ with $0<\mu(A_n)<\infty$ for all $n$. Suppose $\fB$ is directed. Then $D$ is directed and $\|\cdot\|$ is absolutely dominating. 
\end{corollary}
\begin{proof}
This follows from Theorem \ref{theorem:directed_N_to_D_Bochner_implies_norm_bound_useful} since $f\mapsto \sumn f(n) \1_{A_n}$ forms an isometric order preserving isomorphism from the Bochner integrable functions $\N \rightarrow D$ into $\fB$.
\end{proof}

\begin{obs}
\label{obs:bochner_directed_finite_atoms}
Whenever there do not exist $A_1,A_2,\dots$ as in Corollary \ref{cor:bochner_directed_when_countable_pos_sets_implies_norm_bound}, then $X= A_1 \cup \cdots \cup A_N$, where $A_1,\dots, A_N$  are disjoint atoms. 
Let $\alpha_n = \mu(A_n) \in (0,\infty)$. 
Define a norm $\|\cdot\|_N$ on $D^N$ by $\|x\|_N = \sum_{n=1}^N \|x_n\|$. 
Then $T: \fB \rightarrow D^N$ defined by $T(\sum_{n=1}^N x_n \1_{A_n}) = (\alpha_1 x_1, \dots, \alpha_N x_N)$, is an isometric order preserving isomorphism. Therefore,
\begin{align}
D^+ \mbox{ is generating } \iff (D^N)^+ \mbox{ is generating } \iff  \fB^+ \mbox{ is generating}.
\end{align}
\end{obs}

\begin{theorem}
\label{theorem:closed_and_generating_for_bochner_and_space}
 $D^+$ is closed and generating if and only if $\fB^+$ is. 
Moreover, if $D^+$ is closed and generating and $C>0$, then $\|\cdot\|$ is $C$-absolutely dominating if and only if $\|\cdot\|_{\fB}$ is. 
\end{theorem}
\begin{proof}
This follows from Lemma \ref{lemma:norm_bound_for_bochner} and Theorem \ref{theorem:same_bound_norm_space_and_bochner}. 
%
%
\end{proof}

\section{An extension of the Bochner integral}
\label{section:generalised_Bochner}

We present the definition of a Banach cover (Definition \ref{def:Banach_cover}), some examples, and use this notion to extend the Bochner integral to functions with values in such a vector space (\ref{obs:definition_B_integral}). 

The next result follows by definition of the Bochner integral or by \ref{obs:facts_on_Bochner}\ref{item:Bochner_and_closed_map}. 

\begin{theorem}
\label{theorem:integrals_subbanach_spaces_agree}
Let $D_1$ and $D$ 
be Banach spaces and suppose $D_1\subset D$ and the inclusion map is continuous. 
Suppose $f: X \rightarrow D$ has values in $D_1$ and is Bochner integrable as map $X\rightarrow D_1$ with integral $I$. Then $f$ is Bochner integrable as a map $X \rightarrow D$ with integral $I$. 
\end{theorem}

\begin{definition}
\label{def:Banach_cover}
Let $E$ be a vector space. 
Suppose $\B$ is a collection of Banach spaces whose underlying vector spaces are linear subspaces of $E$. 
$\B$ is called a 
\emph{Banach cover}
of $E$ if $\bigcup \B = E$ and 
for all $D_1$ and $D_2$ in $\B$ there exists a  $D\in \B$ with $D_1, D_2\subset D$ such that both inclusion maps $D_1\rightarrow D$ and $D_2 \rightarrow D$ are continuous. 

For a $D\in \B$, we write $\|\cdot\|_D$ for its norm if not indicated otherwise. 

If $E$ is an ordered vector space, then an \emph{ordered Banach cover} is a Banach cover whose elements are seen as ordered subspaces of $E$. 
\end{definition}

\begin{obs}
\label{obs:pedantry_about_norm}
A bit of pedantry: strictly speaking, a Banach space is a couple $(D,\|\cdot\|)$ consisting of a vector space $D$ and a norm $\|\cdot\|$. 
One usually talks about the ``Banach space $D$'', the norm being understood. 
Mostly, we adopt that convention but not always. 
In the context of the above definitions one has to be careful. 
A Banach cover may contain several Banach spaces with the same underlying vector space, so that a formula like ``$D_1,D_2 \subset D$'' really is ambiguous. 
What we mean is only an inclusion relation between the vector spaces and no connection between the norms is assumed a proiri. 

However, suppose $(D_1,\|\cdot\|_1)$ and $(D_2,\|\cdot\|_2)$ are elements of a Banach cover $\B$ and $D_1=D_2$. 
There is a Banach space $(D,\|\cdot\|_D)$ in $\B$ with $D_1, D_2 \subset D$ and with continuous inclusion maps. 
If a sequence $(x_n)_{n\in\N}$ in $D_1 (=D_2)$ is $\|\cdot\|_1$-convergent to $a$ and $\|\cdot\|_2$-convergent to $b$, then it is $\|\cdot\|_D$-convergent to $a$ and $b$, so $a=b$. 
Hence, by the Closed Graph Theorem the norms $\|\cdot\|_1$ and $\|\cdot\|_2$ are equivalent. 

Similarly, if $(D_1,\|\cdot\|_1)$ and $(D_2,\|\cdot\|_2)$ are elements of a Banach cover and $D_1 \subset D_2$, then the inclusion map $D_1 \rightarrow D_2$ automatically is $\|\cdot\|_1$-$\|\cdot\|_2$-continuous. 
\end{obs}

\begin{example}
\label{example:principal_ideals_banach_cover}
Let $E$ be a  uniformly complete Riesz space. 
The set of principal ideals $\B=\{(E_u,\|\cdot\|_u): u\in E^+, u \ne 0\}$ is an ordered Banach cover of $E$: for $u,v \in E^+$ with $u,v>0$ one has $E_u, E_v \subset E_{u+v}$ and $\|\cdot\|_{u+v}\le \|\cdot\|_u$ on $E_u$. 
\end{example}

\begin{example}
\label{example:measurable_functions_and_positive_cover}
Let $(Y,\cB,\nu)$ be a complete $\sigma$-finite measure space. 
Let $M$ be the space of classes of measurable functions $Y \rightarrow \R$.
A function $\rho : M\rightarrow [0,\infty]$ is called an 
\emph{function norm} if (i) $\rho(f) =0 \iff f=0$ a.e., (ii) $\rho(\alpha f) = |\alpha| \rho(f)$ (where $0 \cdot \infty =0$), (iii) $\rho(|f|) =\rho(f)$, (iv) $\rho(f+g) \le \rho(f) + \rho(g)$, (v) $0\le f \le g$ a.e. implies $\rho(f) \le \rho(g)$ for $f,g\in M$ and $\alpha \in \R$. 
For such function norm $\rho$ the set
\begin{align}
L_\rho = \{ f\in M : \rho(f) <\infty\}
\end{align}
is a normed Riesz space called a \emph{K\"{o}the space} (see \cite[Ch. III \S 18]{JoRo77IntroRiesz} or \cite[Ch. 1 \S 9]{LuZa71}). 

A K\"{o}the space $L_\rho$ is complete, i.e., a Banach lattice if and only if for all $u_n \in L_\rho^+$ with $\sumn \rho(u_n) <\infty$,  $\sumn u_n$ is an element of $L_\rho$ (\cite[Theorem 19.3]{JoRo77IntroRiesz}).  
Examples of complete K\"{o}the spaces are Orlicz spaces (\cite[\S 20]{JoRo77IntroRiesz}). 
In particular, the Banach spaces $L^p(\nu)$ for $p \ge 1$ are K\"{o}the spaces. 
We introduce other examples: 

For $w\in M^+$ with $w>0$ a.e. define $\rho_w : M \rightarrow [0,\infty]$ by 
\begin{align}
\rho_w(f) := \int |f| w  \D \nu.
\end{align}
Then $\rho_w$ is a function norm. 
Both the set $\{ L_{\rho_w} : w\in M^+, w>0 \mbox{ a.e.}\}$ 
and the set of all complete K\"{o}the spaces are Banach covers of $M$ (see Appendix \ref{section:appendix_kothe_spaces}). 
\end{example}

\begin{obs}
\label{obs:definition_B_integral}
Let $E$ be a vector space with a Banach cover $\B$. 
Let $(X,\cA,\mu)$ be a complete $\sigma$-finite measure space, $\mu \ne 0$. 
\begin{enumerate}[label=(\arabic*)]
\item For $D\subset \B$ denote by $B_D$ the vector space of all Bochner integrable functions $X\rightarrow D$, and, by $b_D$ the Bochner integral $B_D \rightarrow D$. 
\item 
\label{item:almost_everywhere_equal_functions_with_values_in_different_B_spaces}
Let $D_1,D_2\in \B$, $f_1 \in B_{D_1}$, $f_2\in B_{D_2}$, $f_1 = f_2$ $\mu$-a.e. 
Then $b_{D_1} (f_1) = b_{D_2}(f_2)$. 
\emph{Proof.} Choose $D\in \B$ as in Definition \ref{def:Banach_cover}. 
Then $f_1,f_2\in B_D$ and $b_{D_1}(f_1) = b_D (f_1) = b_D (f_2) = b_{D_2}(f_2)$. 
\item If $D_1,D_2\in \B$ have the same underlying vector space, then $B_{D_1}=B_{D_2}$ since the identity map $D_1 \rightarrow D_2$ is a homeomorphism (see \ref{obs:pedantry_about_norm}).
\item We call a function $f: X \rightarrow E$ \emph{$\B$-integrable} if there is a $D\in \B$ such that $f$ is $\mu$-a.e. equal to some element of $B_D$. 
\item By $\fU$ we indicate the vector space of all $\mu$-equivalence classes of $\B$-integrable functions. 
\end{enumerate}
For $D\in \B$ we have a natural map $T_D: B_D \rightarrow \fU$, assigning to every element of $B_D$ its $\mu$-equivalence class. 
The space $T_D(B_D)$ is a Banach space, which we indicate by $\fB_D$.\footnote{Even though we use the same notation as in Section \ref{section:bochner_integral_on_ordered_Banach_spaces}, see \ref{obs:assumptions_space_useful}, the meaning of $\fB_D$ is slightly different.} 
We write $\fb_D$ for the map $\fB_D \rightarrow D$ determined by 
\begin{align}
\fb_D (T_D(f)) = b_D(f) \qquad (f\in B_D). 
\end{align}
$\fU$ is the union of the sets $\fB_D$. 
By \ref{item:almost_everywhere_equal_functions_with_values_in_different_B_spaces} there is a unique $\fu : \fU \rightarrow E$ determined by 
\begin{align}
\fu( f) = \fb_D(f) \qquad (D\in \B, f\in \fB_D). 
\end{align}
The above leads to the following theorem. 
\end{obs}

\begin{theorem}
$\fU$ is a vector space, $\fu$ is linear and $\{\fB_D: D\in \B\}$ is a Banach cover of $\fU$. 
\end{theorem}

\begin{obs}  
\label{obs:trivial_covers} \
\begin{enumerate}[label=(\arabic*)]
\item 
	For any vector space the finite dimensional linear subspaces form a Banach cover. 
\item 
	If $E$ is a Banach space, then $\{E\}$ is a Banach cover. 
\item 
\label{item:E_Banach_and_contained_in_cover}
	If $E$ is a Banach space and $\B$ is a Banach cover with $E\in \B$, then $\fU$ is just the space of (classes of) Bochner integrable functions $X \rightarrow E$ and $\fu$ is the Bochner integral. 
\item 
\label{item:E_itself_has_unit}
	A special case of \ref{item:E_Banach_and_contained_in_cover}: 
	Let $E$ be a uniformly complete Riesz space with a unit $e$ and let $\B$ be the Banach cover of principal ideals as in Example \ref{example:principal_ideals_banach_cover}. Then $(E,\|\cdot\|_e)\in \B$. 
\end{enumerate}
\end{obs}

%

\begin{obs}
\label{obs:inclusion_covers}
Let $E$ be a vector space and $\B_1$ and $\B_2$ be Banach covers of $E$. 
Suppose that for all $D_1 \in \B_1$ there exists a $D_2 \in \B_2$ with $D_1 \subset D_2$ such that the inclusion map is continuous. 
Write $\fU_i, \fu_i$ for the set of $\B_i$-integrable functions and the $\B_i$-integral, for $i\in \{1,2\}$. Then $\fU_1 \subset \fU_2$ and $\fu_1 = \fu_2 $ on $\fU_1$. 
\end{obs}

\begin{example}[Different covers and different integrals]
\label{example:two_norms_on_Banach_space_give_different_integrals}
Let $(E,\|\cdot\|)$ be an infinite dimensional Banach space. 
Let $T: E \rightarrow E$ a linear bijection that is not continuous; say, there exist $x_n \in E$ with $\|x_n\| = 2^{-n}$, $Tx_n = x_n$, $T(\sumn x_n) \ne \sumn x_n$. 
Define $\|x\|_T = \|Tx\|$ for $x\in E$. 
Then $(E,\|\cdot\|_T)$ is a Banach space. 
The map $f:\N \rightarrow E$ given by $f(n)= x_n$ is Bochner integrable in $(E,\|\cdot\|)$ and in $(E,\|\cdot\|_T)$, but the integrals do not agree. 
Whence with $\B_1 = \{ (E,\|\cdot\|)\}$ and $\B_2 = \{ (E,\|\cdot\|_T)\}$ we have $f\in \fU_1 \cap \fU_2$ but $\fu_1(f) \ne \fu_2(f)$. 
\end{example}

As an immediate consequence of Theorem \ref{theorem:Bochner_functions_Banach_lattice} we obtain the following theorem. 

\begin{theorem} 
\label{theorem:Banach_lattice_cover_implies_Bochner_Riesz_space}
Suppose $E$ is a Riesz space and $\B$ is a Banach cover of $E$ that consists of Banach lattices that are Riesz subspaces of $E$. 
Then $\fU$ is a Riesz space and $\fu$ is order preserving. 
\end{theorem}
\begin{proof}
This is a consequence of Theorem \ref{theorem:Bochner_functions_Banach_lattice}. 
\end{proof}

\begin{obs}
\label{obs:order_preserving_and_covers}
Whenever $E$ is an ordered vector space and $\B$ an ordered Banach cover of $E$, then $\fU$ is an ordered vector space. 
In order for $\fu$ to be order preserving, one needs a condition on $\B$. 
This and other matters will be treated in \S\ref{section:Bochner_on_covers2}. 
A sufficient condition turns out to be closedness of $D^+$ for every $D\in \B$ (see Theorem \ref{theorem:closed_cone_implies_order_preservingness_integral} and Theorem \ref{theorem:closed_cones_in_banach_cover_imply_order_preservingness_integral}). 
First we will see in \S\ref{section:useful_covers} 
that  all Archimedean directed ordered vector spaces 
admit such ordered Banach covers. (The Archimedean property is necessary as follows easily from Theorem \ref{theorem:ordering_determined_by_dual_cone}). 
\end{obs}

\begin{obs}
\label{obs:bornological_convergence_and_B_convergence}
Whenever $E$ is a vector space and $\B$ is a Banach cover of $E$, then the set 
\begin{align}
\{ A\subset E: \mbox{ there exists a } D\in \B \mbox{ such that } A \mbox{ is bounded in } D\}
\end{align}
forms a bornology on $E$ (we refer to the book of Hogbe-Nlend \cite{Ho77} for the theory of bornologies). 
\end{obs}

\section{Covers of ordered Banach spaces with closed generating cones}
\label{section:useful_covers}

\textbf{In this section $E$ is an Archimedean directed 
ordered vector space. 
$\B$ is the collection of all ordered Banach spaces that ordered linear subspaces of $E$ whose cones are closed and generating. 
}

We intend to prove that $\B$ is a Banach cover of $E$ (Theorem \ref{theorem:cover_by_ordered_Banach_spaces_with_closed_generating_cone}). 

\begin{lemma}
\label{lemma:same_convergence_on_intersections}
Let $(D_1, \|\cdot\|_1)$, $(D_2,\|\cdot\|_2)$ be in $\B$.
Let $z_1,z_2,\dots \in D_1 \cap D_2$, $a\in D_1$, $b\in D_2$, $\|z_n - a\|_1 \rightarrow 0$, $\|z_n - b\|_2 \rightarrow 0$. Then $a=b$. 
\end{lemma}
\begin{proof}
We may assume $\sumn \|z_n - a\|_1 <\infty$ and $\sumn \|z_n - b\|_2 <\infty$. 
Then $z_n \unifrightarrow a$ in $D_1$ and $z_n \unifrightarrow b$ in $D_2$ by Theorem \ref{theorem:equivalences_useful}. 
Then $z_n \unifrightarrow a$ and $z_n \unifrightarrow b$ in $E$. 
Because $E$ is Archimedean, $a=b$. 
\end{proof}

\begin{obs}
\label{obs:norms_on_ordered_Banach_space_with_closed_generating_cone_are_equivalent}
If $D$ is an ordered Banach space with closed generating cone $D^+$, under each of two norms $\|\cdot\|_1$ and $\|\cdot\|_2$, then these norms are equivalent. Indeed, the identity map $(D,\|\cdot\|_1) \rightarrow (D,\|\cdot\|_2)$ has a closed graph by Lemma \ref{lemma:same_convergence_on_intersections}. 
\end{obs}

\begin{theorem}
\label{theorem:sum_has_also_closed_generating_cone}
Let $(D_1, \|\cdot\|_1)$, $(D_2,\|\cdot\|_2)$ be in $\B$.
$D_1+D_2$ is an ordered Banach space with closed generating cone under the norm 
 $\|\cdot\|: D_1 + D_2 \rightarrow [0,\infty)$ defined by  
\begin{align}
\|z\| := \inf \{ \|x\|_1 + \|y\|_2 : x\in D_1, y\in D_2, z= x+y\}. 
\end{align}
Moreover, if $C>0$ and $\|\cdot\|_1$ and $\|\cdot\|_2$ are $C$-absolutely dominating, then so is  $\|\cdot\|$. 
\end{theorem}
\begin{proof}
$D_1 \times D_2$ is a Banach space under the norm $(x, y) \mapsto \|x\|_1 + \|y\|_2$. 
From Lemma \ref{lemma:same_convergence_on_intersections} it follows that $\Delta:= \{(a,b)\in D_1 \times D_2: a=-b\}$ is closed in $D_1 \times D_2$. 
Then $D_1 \times D_2 / \Delta$ is a Banach space under the quotient norm. 
This means that $D_1 + D_2$ is a Banach space under $\|\cdot\|$.  (In particular, $\|\cdot\|$ is a norm.) \\
Since $D_1^+ + D_2^+ \subset (D_1+D_2)^+$, the latter is generating. \\
We prove that $(D_1+D_2)^+$ is closed. 
Let $u_1,u_2,\dots \in D_1+D_2$, $\sumn \|u_n\|<\infty$; we prove $u_n \unifrightarrow 0$. 
Choose $x_n \in D_1$, $y_n\in D_2$ with $u_n = x_n + y_n$ and $\sumn \|x_n\|_1 <\infty$, $\sumn \|y_n\|_2<\infty$. Then, see Theorem \ref{theorem:equivalences_useful}, $x_n \unifrightarrow 0$ in $D_1$ and $y_n \unifrightarrow 0$ in $D_2$. It follows that $u_n \unifrightarrow 0$ in $D_1+D_2$. By Theorem \ref{theorem:equivalences_useful} it follows that $(D_1+ D_2)^+$ is closed. \\
Suppose $C>0$ is such that $\|\cdot\|_1$ and $\|\cdot\|_2$ are $C$-absolutely dominating. 
Let $z\in D_1 + D_2$, $\epsilon>0$.
Choose $x\in D_1$, $y\in D_2$ with $z= x+y$, $\|x\|_1 + \|y\|_2 \le \|z\|+ \frac{\epsilon}{3}$. 
Choose $a\in D_1^+$ with $-a\le  x \le a$, $\|a\|_1 < C\|x\|_1 + \frac{\epsilon}{3}$ and $b\in D_2^+$ with $-b \le y \le b$, $\|b\|_2< C\|y\|_2 +\frac{\epsilon}{3}$. 
Set $c = a+b$.  
Then $c\in (D_1+D_2)^+$, $-c\le z\le c$ and $\|c\| \le \|a\|_1+ \|b\|_2 < C\|x\|_1 + C\|y\|_2 + 2 \frac{\epsilon}{3} < C\|z\|+\epsilon$.
\end{proof}

\begin{obs}
\label{obs:every_element_is_contained_in_a_useful_BS}
Let $x\in E$. (We make a $D\in \B$ with $x\in D$.)
Choose $a\in E$ such that $-a\le x \le a$. 
Let $D= \R ( a- x) + \R( a+x)= \R a+ \R x$. 
$D$ is a directed 
ordered vector space. 
Define $\|\cdot\|: D \rightarrow [0,\infty)$ by 
\begin{align}
\|y\| := \inf \{ s \ge 0 : -sa \le y \le s a\}.
\end{align}
Then $\|\cdot\|$ is a norm on $D$, $-\|y\|a \le y \le \|y\| a$ for all $y\in D$ and $\|a\|=1$. 
Thus $(D,\|\cdot\|)$ is a directed ordered Banach space. Moreover $D^+$ is closed: Let $y\in D$, $y_1,y_2,\dots \in D^+$, $\|y -y_n\| < \frac1n$. 
Then $y \ge y_n - \frac1n a\ge - \frac1n a$, so $y \ge 0$. 

$\|\cdot\|$ is $1$-absolutely dominating:
$\|a\|=1$, so 
$\inf \{\|c\|: c\in D^+,  -c \le y \le c\} \le \inf \{ s \ge 0: -sa \le y \le sa\} = \|y\|$. 

Even $\|y \| = \inf \{\|c\|: c\in D^+, -c \le y \le c\}$: For $c\in D^+$ with $-c \le y \le c$ and $s\ge 0$ such that $c \le s a$ we have $-sa \le -c \le y \le c \le sa$ and so $\|y\|\le \|c\|$. 
\end{obs}

\begin{theorem}
\label{theorem:cover_by_ordered_Banach_spaces_with_closed_generating_cone}
$\B$ is a Banach cover of $E$.
Moreover, 
\begin{align}
\{ D\in \B : \|\cdot\|_D \mbox{ is $1$-absolutely dominating }\}
\end{align}
 is a Banach cover of $E$. 
\end{theorem}
\begin{proof}
By \ref{obs:every_element_is_contained_in_a_useful_BS} each element of $ E$ is contained in an ordered Banach space with closed generating cone (with a $1$-absolutely dominating norm). 
By Theorem \ref{theorem:sum_has_also_closed_generating_cone} and by definition of the norm, $\B$ forms a Banach cover of $E$. 
\end{proof}

\begin{obs}
It is reasonable to ask if an analogue of Theorem \ref{theorem:cover_by_ordered_Banach_spaces_with_closed_generating_cone} holds in the world of Riesz spaces: does every Archimedean Riesz space have a Banach cover consisting of Riesz spaces?
The answer is negative. 

Let $E$ be the Riesz space of all functions $f$ on $\N$ for which there exist $N\in\N$ and $r,s\in \R$ such that $f(n) = sn + r$ for $n\ge N$. 
Suppose $E$ has a Banach cover $\B$ consisting of Riesz subspaces of $E$. 
There is a $D\in \B$ that contains the constant function $\1$ and the identity map $i : \N \rightarrow \N$. 
For every $n\in\N$, 
\begin{align}
\1_{\{1,\dots,n\}} = \1 \vee (n+1) \1 - i \vee n\1 \in D.
\end{align}
It follows that $D=E$, so $E$ is a Banach space under some norm. 

But $E$ is the union of an increasing sequence $D_1 \subset D_2 \subset \cdots$ of finite dimensional -hence, closed- linear subspaces:
\begin{align}
D_n = \R \1 + \R i + \R \1_{\{1\}} + \cdots +  \R \1_{\{n\}}.
\end{align}
By Baire's Category Theorem, some $D_n$ has nonempty interior in $E$. Then $E=D_n$ and we have a contradiction. 
\end{obs}

\begin{obs}
In Theorem \ref{theorem:cover_by_ordered_Banach_spaces_with_closed_generating_cone} we single out one particular Banach cover $\B$. 
If we consider only Banach covers consisting of directed spaces with closed cones, this $\B$ is the largest and gives us the largest collection of integrable functions. 
Without directedness there may not be a largest Banach cover. For instance, consider Example \ref{example:two_norms_on_Banach_space_give_different_integrals}. 
Impose on $E$ the trivial ordering ($x\le y$ if and only if $x=y$). 
Then $E^+=\{0\}$, and both $\B_1$ and $\B_2$ consist of Banach spaces with closed (but not generating) cones. 
\end{obs}

\section{The integral for an Archimedean ordered vector space}
\label{section:Bochner_on_covers2}

As a consequence of Theorem \ref{theorem:closed_cone_implies_order_preservingness_integral} we obtain the following extension of Theorem \ref{theorem:Banach_lattice_cover_implies_Bochner_Riesz_space}.

\begin{theorem}
\label{theorem:closed_cones_in_banach_cover_imply_order_preservingness_integral}
Let $E$ be an ordered vector space with an ordered Banach cover $\B$ so that $D^+$ is closed for all $D\in \B$. 
Then $\fu$ is order preserving.
Moreover, $E$ and $\fU$ are Archimedean. 
\end{theorem}

\begin{lemma}
\label{lemma:map_into_vector_space_giving_a_useful_space}
Let $D$ be an ordered Banach space with a closed generating cone $D^+$. Let $T$ be a linear order preserving map of $D$ into an Archimedean ordered vector space $H$. 
Then $\ker T$ is closed and $T(D)$ equipped with the norm $\|\cdot\|_q$ given by 
\begin{align}
\label{eqn:norm_on_image_T}
\|z\|_q = \inf \{ \|x\| : x\in D, Tx =z\},
\end{align}
has a closed generating cone $T(D)^+$. 
\end{lemma}
\begin{proof}
(I) Let $x_1,x_2,\dots \in \ker T$, $x\in D$, $x_n \rightarrow x$; we prove $x\in \ker T$. 
We assume $\sumn \|x- x_n\| <\infty$. 
By Theorem \ref{theorem:equivalences_useful} $x_n \xrightarrow{u} x$. 
By Lemma \ref{lemma:order_bounded_linear_implies_continuous_wrt_ru} $Tx_n \xrightarrow{u} Tx$, so $Tx=0$. \\
(II) $D/ \ker T$ is a Banach space under the quotient norm $\|\cdot\|_Q$. 
The formula $x+ \ker T \mapsto Tx$ describes a linear bijection $D/\ker T \rightarrow T(D)$ and 
\begin{align}
\|x + \ker T\|_Q = \|Tx\|_q \qquad (x\in D). 
\end{align}
It follows that $\|\cdot\|_q$ is indeed a norm, turning $T(D)$ into a Banach space. 

$T(D^+) \subset T(D)^+$, whence $T(D)$ is directed. 

We prove that $T(D)$ satisfies \ref{item:condition_ru_convergence_summable_seq} of Theorem \ref{theorem:equivalences_useful}: 
Let $z_1,z_2,\dots \in T(D)$ be such that $\sumn \|z_n\|_q<\infty$. 
Choose $x_n\in D$ such that $T x_n = z_n$, $\sumn \|x_n\|<\infty$. 
Using \ref{item:condition_ru_convergence_summable_seq} for $D$, $x_n \unifrightarrow 0$. Then $z_n= Tx_n \unifrightarrow 0$ by Lemma \ref{lemma:order_bounded_linear_implies_continuous_wrt_ru}.
\end{proof}

\begin{obs}
In the proof of Lemma \ref{lemma:map_into_vector_space_giving_a_useful_space} we mentioned the inclusion $T(D^+) \subset T(D)^+$. 
This inclusion can be strict. 
Take $D= H=\R^2$, $T(x,y) = (x, x+y)$. 
Then $T(D^+) \ne T(D)^+$. 
\end{obs}

From Theorem \ref{theorem:order_bounded_linear_between_useful_implies_continuous}, \ref{obs:facts_on_Bochner}\ref{item:Bochner_and_closed_map} and Lemma \ref{lemma:map_into_vector_space_giving_a_useful_space} we get:

\begin{theorem}
\label{theorem:composition_order_bounded_linear_map}
Let $E_1, E_2$ be ordered vector spaces, $E_i$ endowed with the Banach cover $\B_i$ consisting of the ordered Banach spaces with closed generating cones. 
Let $T: E_1 \rightarrow E_2$ be linear and order preserving. 
If $f: X \rightarrow E_1$ is $\B_1$-integrable, then $T\circ f: X \rightarrow E_2$ is $\B_2$-integrable, and $\fu_2( T \circ f ) = T( \fu_1( f ))$. 
\end{theorem}

\begin{obs}
In view of Theorem \ref{theorem:order_bounded_linear_between_useful_implies_continuous} 
the reader may wonder why in Theorem \ref{theorem:composition_order_bounded_linear_map} $T$ is required to be order preserving and not just order bounded, 
the more so because of the following considerations. 
Let $D$ and $H$ be as in Lemma \ref{lemma:map_into_vector_space_giving_a_useful_space} and $T$ be a linear order bounded map of $D$ into $H$. 
As the implication \ref{item:archi_and_ru_condition} $\Longrightarrow$ \ref{item:useful} of Theorem \ref{theorem:equivalences_useful} is valid for Archimedean (but not necessarily directed) $D$, 
following the lines of the proof of Lemma \ref{lemma:map_into_vector_space_giving_a_useful_space} $\ker T$ is closed and $T(D)$ equipped with the norm as in \eqref{eqn:norm_on_image_T} has a closed cone $T(D)^+$. 
However, we also need $T(D)$ to be directed and order boundedness of $T$ is no guarantee for that. 

An alternative approach might be to drop the directedness condition on the spaces that constitute $\B$. 
However, the ordered Banach spaces with closed cones may not form a Banach cover. 

For an example, let $E$ be $\ell^\infty$ and let $\B$ be the collection of all ordered Banach spaces that are subspaces of $\ell^\infty$ and have closed cones. 
We make $D_1,D_2\in \B$. 
For $D_1$ we take $\ell^\infty$ with the usual norm $\|\cdot\|_\infty$. 
Choose a linear bijection $T: \ell^\infty \rightarrow \ell^\infty$ that is not continuous. 
For $a\in \ell^\infty$ put $a' = (a_1,-a_1,a_2,-a_2,\dots)$. 
For $D_2$ we take the vector space $\{ a' : a\in \ell^\infty\}$ with the norm $\|\cdot\|_T$ given by $\|a'\|_T= \|Ta\|_\infty$. 
Then $D_2$ is a Banach space and $D_2^+$, begin $\{0\}$, is closed in $D_2$. 
Suppose $\B$ is a Banach cover. 
Let $D$ be as in Definition \ref{def:Banach_cover}. 
By the Open Mapping Theorem the identity map $D_1 \rightarrow D$ is a homeomorphism. 
By the continuity of the inclusion map $D_2 \rightarrow D$ there exists a number $c$ such that $\|a'\|_T \le c\|a'\|_\infty$ for all $a\in \ell^\infty$. 
Then $\|Ta\|_\infty \le c\|a'\|_\infty \le c\|a\|_\infty$ for $a\in \ell^\infty$, so $T$ is continuous. Contradiction. 
\end{obs}

\begin{theorem}
\label{theorem:separating_order_dual_and_pettis_like_integration}
Let $E$ be an ordered vector space such that $E^\sim$ separates the points of $E$.\footnote{We write $E^\sim$ for the space of order bounded linear maps $E\rightarrow \R$.} 
Assume $\B$ is a Banach cover consisting of ordered Banach spaces with closed generating  cones. 
Let $f\in \fU$. Then $\alpha \circ f \in \cL^1(\mu)$ for all $\alpha \in E^\sim$. 
Moreover, $I\in E$ is such that 
\begin{align}
\label{eqn:composition_order_dual_element}
\alpha(I) = \int \alpha \circ f \D \mu \ \mbox{ for all } \ \alpha \in E^\sim,
\end{align}
if and only if $I=\fu(f)$. 
\end{theorem}
\begin{proof}
Let $f\in \fU$ and let $I= \fu(f)$. 
By Theorem \ref{theorem:composition_order_bounded_linear_map} $\alpha \circ f \in \cL^1(\mu)$ for all $\alpha \in E^\sim$ and \eqref{eqn:composition_order_dual_element} holds. 
$I=\fu(f)$ is the only element of $E$ for which \eqref{eqn:composition_order_dual_element} holds because $E^\sim$ separates the points of $E$. 
\end{proof}

\begin{remark}
Functions with values in a Banach space that are Bochner integrable are also Pettis integrable. 
To some extent the statement of Theorem \ref{theorem:separating_order_dual_and_pettis_like_integration} is similar. 
Indeed, the definition of Pettis integrability could be generalised for vector spaces $V$ which are equipped with a set $S$ of linear maps $V \rightarrow \R$ that separates the points of $V$, in the sense that one calls a function $f: X \rightarrow V$ Pettis integrable if $\alpha \circ f \in \cL^1(\mu)$ for all $\alpha \in S$ and there exists a $I\in V$ such that $\alpha(I)= \int \alpha \circ f \D \mu$ for all $\alpha \in S$. 
Then Theorem \ref{theorem:separating_order_dual_and_pettis_like_integration} implies that every $f\in \fU$ is Pettis integrable when considering $V=E$ and $S=E^\sim$.
Observe, however, that even for a Riesz space $E$, $E^\sim$ may be trivial (see, e.g., \cite[5.A]{JoRo77IntroRiesz}). \todolater{evt andere citatie?}
\end{remark}

\section{Comparison with other integrals}
\label{section:comparing_integrals}

\textbf{In this section $(X,\cA,\mu)$ is a complete $\sigma$-finite measure space and $E$ is a directed ordered vector space with an ordered Banach cover $\B$ so that $D^+$ is closed for each $D\in \B$. 
}


In \ref{obs:definition_B_integral} we have introduced an integral $\fu$ on a space $\fU$ of $\B$-integrable functions $X \rightarrow E$.\footnote{In this section we close an eye for the difference between a function and its equivalence class. There will be no danger of confusion.}
In \cite{vRvZ}, starting from a natural integral $\varphi$ on the space $S$ of all simple functions $X\rightarrow E$ we have made integrals $\varphi_V, \varphi_L, \varphi_{LV}, \dots$ on spaces $S_V, S_L, S_{LV},\dots$. 

There is an elementary connection: $S$ is part of $\fU$ and $\fu$ coincides with $\varphi$ on $S$. (Indeed, let $f\in S$. Being a finite set, $f(X)$ is contained in $D$ for some $D\in \B$. Then $f$ is Bochner integrable as a map $X\rightarrow D$.) 

In general, $S_V$ and $S_L$ are not subsets of $\fU$, but we can prove that $\fu$ coincides with $\varphi_V$ on $S_V \cap \fU$ and with $\varphi_L$ on $S_L \cap \fU$. 
Better than that: $\fu$ is ``compatible'' with $\varphi_V$ in the sense that $\fu$ and $\varphi_V$ have a common order preserving linear extension $S_V + \fU \rightarrow E$. Similarly, $\fu$ is ``compatible with $\varphi_L$, $\varphi_{LV},\dots$''. 

\begin{lemma}
\begin{enumerate}
\item Let $f\in \fU$, $g\in S_{V}$, $f\le g$. Then $\fu(f) \le \varphi_V(g)$. 
\label{item:fU_function_and_V_function_integrals_order_compared}
\item Let $f\in \fU$, $g\in S_{L}$, $f\le g$. Then $\fu(f) \le \varphi_L(g)$. 
\label{item:fU_function_and_L_function_integrals_order_compared}
\end{enumerate}
\end{lemma}
\begin{proof}
\ref{item:fU_function_and_V_function_integrals_order_compared}
By the definition of $\varphi_V$ and by the text preceding this lemma 
we have $\varphi_V(g) = \inf \{ \varphi(h) : h\in S, h \ge g \} = \inf \{ \fu(h) : h\in S, h\ge g\}$. 
As $g \ge f$ and $\fu$ is order preserving (Theorem \ref{theorem:closed_cones_in_banach_cover_imply_order_preservingness_integral}), it follows that $\varphi_V(g) \ge \fu(f)$. 

\ref{item:fU_function_and_L_function_integrals_order_compared}
Let $g\in S_L$ and assume $f\le g$. 
Let $g_1,g_2\in S_L^+$ be such that $g= g_1 - g_2$. 
Let $(B_i)_{i\in\N}$ be a $\varphi$-partition for both $g_1$ and $g_2$. 
Write $A_n = \bigcup_{i=1}^n B_i$ for $n\in\N$. 
Then $f\1_{A_n} \le g\1_{A_n}$, thus by (a) (and Theorem \ref{theorem:closed_cones_in_banach_cover_imply_order_preservingness_integral})
\begin{align}
\fu(f\1_{A_n}) 
\notag &\le \fu (g \1_{A_n}) = \varphi (g\1_{A_n}) 
= \varphi(g_1 \1_{A_n}) - \varphi(g_2 \1_{A_n}) \\
& \le \varphi_L(g_1) - \varphi(g_2 \1_{A_k }) \qquad (k\in\N, k<n).
\end{align}
Which implies $\fu(f\1_{A_n}) + \varphi(g_2\1_{A_k}) \le  \varphi_L(g_1)$ for all $k<n$. Then letting $n$ tend to $\infty$ (apply \ref{obs:facts_on_Bochner}\ref{item:dominated_convergence_Bochner}: $f(x)\1_{A_n}(x) \rightarrow f(x)$ for all $x\in X$) we obtain
\begin{align}
\fu(f)  \le \varphi_L(g_1) - \varphi(g_2\1_{A_k}) \qquad (k\in\N),
\end{align}
from which we conclude $\fu(f) \le \varphi_L(g)$. 
\end{proof}

\begin{theorem}
\label{theorem:comparing_S_LV_with_fU} \
\begin{enumerate}
\item If $g\in S_{LV}$ and $f\le g$, then $\fu(f) \le \varphi_{LV}(g)$. 
\item If $S_V$ is stable,  $g\in S_{VLV}$ and $f\le g$, then $\fu(f) \le \varphi_{VLV}(g)$. 
\end{enumerate}
\end{theorem}
\begin{proof}
Follow the lines of the proof of the lemma with $S_V$, $S_L$ or $S_{VL}$ instead of $S$. 
\end{proof}

\begin{obs}[Comments on Theorem \ref{theorem:comparing_S_LV_with_fU}] \
\begin{enumerate}[label=(\arabic*)]
\item The theorem supersedes the lemma because $S_V + S_L \subset S_{LV}$. 
\item As a consequence, $\fu = \varphi_{LV}$ on $\fU \cap S_{LV}$, and $\fu= \varphi_{VLV}$ on $\fU \cap S_{VLV}$ if $S_V$ is stable. 
\item Recall that stability of $S_V$ is necessary for the existence of $S_{VLV}$. 
\end{enumerate}
\end{obs}

\begin{theorem}
\label{theorem:integrable_functions_part_of_LV_for_uniformly_complete_Riesz_space}
Let $E$ be a uniformly complete Riesz space and $\B$ be the Banach cover of principal ideals (see Example \ref{example:principal_ideals_banach_cover}).
$\fU$ is a linear subspace of $S_{LV}$ and $\fu = \varphi_{LV}$ on $\fU$. 
\end{theorem}
\begin{proof}
Let $u\in E^+$ and let $f: X\rightarrow E_u$ be Bochner integrable. 
We prove $f\in S_{LV}$.

For simplicity of notation, put $D=E_u$. 
Let $S^D$ be the space of simple functions $X\rightarrow D$.
By \cite[Corollary 9.8]{vRvZ} we have $f\in (S^D)_{LV}$ and $\fu(f) = \varphi_{LV}(f)$. 
Since $D$ is a Riesz ideal in $E$, the identity map $D\rightarrow E$ is order continuous. 
Then \cite[Theorem 8.14]{vRvZ} implies $f\in S_{LV}$. 
\end{proof}

Contrary to Theorem \ref{theorem:integrable_functions_part_of_LV_for_uniformly_complete_Riesz_space}, in \cite[Example 9.9(II)]{vRvZ} $\fU$ is not a linear subspace of $S_{LV}$. Then next example shows, in the context of Theorem \ref{theorem:integrable_functions_part_of_LV_for_uniformly_complete_Riesz_space}, that the inclusion may be strict. 

\begin{example}[$\fU \subsetneq S_{LV}$]
For $X=\N, \cA= \cP(\N)$ and $\mu$ the counting measure and $E=c$. 
As is mentioned in \cite[Examples 9.9(I)]{vRvZ}, the function $n\mapsto \1_{\{n\}}$ is an element of $S_{LV}$ but not Bochner integrable. 
With $\B$ the Banach cover of principal ideals, the function $n \mapsto \1_{\{n\}}$ is not $\B$-integrable 
(see \ref{obs:trivial_covers}\ref{item:E_itself_has_unit}). 
\end{example}

\section{An example: Convolution}
\label{section:convolution}

To illustrate the $\B$-integral as an extension of the Bochner integral we consider the following situation. 
(This introduction requires some knowledge of harmonic analysis on locally compact groups, the balance of this section does not.)

Let $G$ be a locally compact group. 
For $f: G\rightarrow \R$ and $x\in G$ we let $L_x f : G \rightarrow \R $ be the function $y\mapsto f(x^{-1}y)$. 

For a finite measure $\mu$ on $G$ and $f$ in $L^1(G)$ one defines their convolution product to be the element $\mu * f$ of $L^1(G)$ given for almost every $y\in G$ by 
\begin{align}
(\mu * f)(y) = \int f(x^{-1} y ) \D \mu(x) = \int (L_x f)(y) \D \mu(x). 
\end{align}
The map $x\mapsto L_x f$ of $G$ into $L^1(G)$ is continuous and bounded, hence Bochner integrable with respect to $\mu$. It is not very difficult to prove that 
\begin{align}
\mu * f = \int L_x f \D \mu(x). 
\end{align}

Similar statements are true for other spaces of functions instead of $L^1(G)$, such as $L^p(G)$, with $1<p<\infty$, and $C_0(G)$, the space of continuous functions that vanish at infinity. 

But consider the space $C(G)$ of all continuous functions on $G$. 
The integrals \\
$\int f(x^{-1}y) \D \mu(x)$ will not exist for all $f\in C(G)$, $y\in G$ and all finite measures $\mu$, but they do if $\mu$ has compact support. 
Thus, one can reasonably define $\mu * f$ for $f\in C(G)$ and compactly supported $\mu$. 
However, there is no natural norm on $C(G)$ (except, of course, if $G$ is compact), so we cannot speak of $\int L_x f \D \mu(x)$ as a Bochner integral. 
We will see that, at least for $\sigma$-compact $G$, it is a $\B$-integral where $\B$ is the Banach cover of $C(G)$ that consists of the principal (Riesz) ideals. 

 

\begin{theorem}
\label{theorem:left_translation_in_principal_ideal}
Let $G$ be a $\sigma$-compact locally compact group. 
For every $f\in C(G)$ there exists a $w\in C(G)^+$ such that 
\begin{align}
\notag &\mbox{Every } L_x f \quad (x\in G) \mbox{ lies in the principal ideal } C(G)_w; \\
& x\mapsto L_x f \mbox{ is continuous relative to } \|\cdot\|_w. 
\label{eqn:left_translation_in_principal_ideal_and_continuity}
\end{align}
\end{theorem}
\begin{proof}
Choose compact $K_1 \subset K_2 \subset \dots $ such that 
$K_1$ is a neighbourhood of $e$; 
$K_n = K_n^{-1}$; $K_n K_n \subset K_{n+1}$, $\bigcup_{n\in\N} K_n = G$. 

For $x\in G$, define $[x]$ to be the smallest $n$ with $x\in K_n$. 
Then $[x]=[x^{-1}]$, $[xy]\le 1 + [x]\vee [y]$ for all $x,y\in G$ by definition. 

Let $f\in C(G)$. Define $u, \v : G \rightarrow [1,\infty)$ as follows: 
\begin{align}
u(x) : = 1 + \sup |f(K_{n+1})|,\quad  \v(x) = [x] u(x) \qquad \mbox{ if } [x]=n. 
\end{align}

(1) If $x,y\in G$, $[x]\le [y]=n$, then $x,y\in K_n$, so 
\begin{align}
|f(xy) | \le \sup |f(K_{n+1})|\le u(y). 
\end{align}

(2) Hence, for all $x,y\in G$: $|f(x^{-1}y)|\le u(x) \vee u(y) \le u(x) \v(y)$, i.e., $|L_x f| \le u(x) \v$. 

\bigskip

\noindent Let $a\in G$, $\epsilon>0$. In (5), using (3) and (4), we show the existence of a neighbourhood $U$ of $e$ with 
\begin{align}
x\in a U \Longrightarrow | L_x f - L_a f | \le \epsilon \v. 
\end{align}

Choose a $p$ with $a \in K_p$. 

(3) $K_1$ contains an open set $V$ containing $e$. We make a $q\in \N$ with 
\begin{align}
x\in a V  \Longrightarrow | L_x f - L_a f | \le \epsilon \v \mbox{ on } G \setminus K_q. 
\end{align}
Let $x\in aV$. 
Then $x,a \in K_{p+1}$, so $[x],[a]\le p+1$. By (1): 
\begin{align}
| (L_x f - L_a f)(y)  | \le |f(x^{-1}y)| + |f(a^{-1}y)| \le 2 u (y) \mbox{ if } [y] \ge p+1.
\end{align}
Moreover, $\epsilon \v(y) = \epsilon [y] u(y) \ge 2u (y)$ if $[y] \ge \frac{2}{\epsilon}$. 
Take $q\in \N$ with $q\ge p+1$ and $q\ge \frac{2}{\epsilon}$. For $y\in G \setminus K_q$ we have $[y]>q$, so $|(L_x f - L_a f)(y)| \le \epsilon \v(y)$. 

(4) We show there exists an open set $W$ containing $e$ with 
\begin{align}
\label{eqn:bound_on_K_q}
x\in aW \Longrightarrow | L_x f - L_a f | \le \epsilon \v \mbox{ on } K_q. 
\end{align}
The function $(x,y) \mapsto | (L_x f - L_a f)(y) |$ on $G\times G$ is continuous and $K_q$ is compact. 
Hence by \cite[Theorem 8.15]{vRvZ} the function $G\mapsto [0,\infty)$
\begin{align}
x\mapsto \sup_{y\in K_q} | (L_x f - L_a f)(y) |
\end{align}
is continuous. Its value at $a$ is $0$, so there exists an open set $W$ containing $e$ with 
\begin{align}
\sup_{y\in K_q} | (L_x f - L_a f)(y)| \le \epsilon \qquad (x\in a W). 
\end{align}
As $v(y) \ge 1$ for all $y$ we obtain \eqref{eqn:bound_on_K_q}. 

(5) With $U = V\cap W$ we have
\begin{align}
x\in aU \ \Longrightarrow \ |L_x f - L_a f|\le \epsilon \v \mbox{ on } G. 
\end{align}
Therefore, to prove the theorem it is sufficient to show there exists a continuous function $w \ge \v$: 

(6) Set $\alpha_n = n ( 1+ \sup | f(K_{n+1})| )$ for all $n$; then 
\begin{align}
[x]=n \Longrightarrow \v(x)= \alpha_n.
\end{align}
Put $K_0 = \emptyset$. For all $n\in\N$ we have $V K_{n-1} \subset K_1 K_{n-1} \subset K_n$, so $K_{n-1}$ is a subset of $K_n^\circ$, the interior of $K_n$. 
By Urysohn \cite[Theorem VII.4.1 and Theorem XI.1.2]{Du66} for all $n$ 
there is a continuous $g_n : G \rightarrow [0,1]$ with 
\begin{align}
g_n = 0 \mbox{ on } K_{n-1}, \ g_n = 1 \mbox{ on } G\setminus K_n^\circ \supset G \setminus K_n. 
\end{align}
Let $b\in G$. There is a $l$ with $b\in K_l$. 
$bV$ is an open set containing $b$. As $bV \subset K_{l+1}$ and $g_n=0$ on $K_{n-1}$ we have: $g_n=0$ on $bV$ as soon as $n\ge l+2$. 

Hence, $w:= \sumn \alpha_{n+1} g_n$ is a continuous function $G \rightarrow [0,\infty)$. 
For every $x\in G$ there is an $n$ with $[x]=n$; then $x\notin K_{n-1}$, $g_{n-1}(x)=1$ and $w(x) \ge \alpha_n = \v(x)$. 
\end{proof}

\begin{theorem}
\label{theorem:convolution}
Let $G$ be a $\sigma$-compact locally compact group. 
Let $\mu$ be a finite measure on the Borel $\sigma$-algebra of $G$ with a compact support. 
Let $\B$ be the Banach cover of $C(G)$ consisting of the principal ideals as in Example \ref{example:principal_ideals_banach_cover}. 
Then for every $f\in C(G)$ the function $x\mapsto L_x f$ is $\B$-integrable and its integral is the 
``convolution product'' $\mu * f$: 
\begin{align}
(\mu * f) (y)  = \int f(x^{-1}y) \D \mu(x) \qquad (y\in G). 
\end{align}
\end{theorem}
\begin{proof}
By Theorem \ref{theorem:left_translation_in_principal_ideal} there exists a $w\in C(G)^+$ such that 
\eqref{eqn:left_translation_in_principal_ideal_and_continuity} holds. 
This implies that the map $x \mapsto L_x f$ is Borel measurable and $\{L_x f : x\in G\}$ is separable in $C(G)_w$. As $x\mapsto \|L_x f\|_w$ is continuous and thus bounded on the support of $\mu$, the map $x \mapsto L_x f$ is $\B$-integrable (see \ref{obs:facts_on_Bochner}\ref{item:Bochner_integrable_when_measurable_and_norm_integrable}).  
That the integral is equal to $\mu * f$ follows by \ref{obs:facts_on_Bochner}\ref{item:Bochner_and_closed_map}. 
\end{proof}

\begin{remark}
Theorem \ref{theorem:convolution} compares to \cite[Example 8.17]{vRvZ} in the sense that in both situations the convolution is equal to an integral of the translation. 
Though the situation is slightly different in the sense that in Theorem \ref{theorem:convolution} we consider $\sigma$-compact locally compact groups, while in \cite[Example 8.17]{vRvZ} we considered metric locally compact groups (the fact that in \cite[Example 8.17]{vRvZ} $L_x f (t) = f(t x^{-1})$ is reminiscent for its statement). 
\end{remark}


\section*{Acknowledgements}

W.B. van Zuijlen is supported by the ERC Advanced Grant VARIS-267356 of Frank den Hollander.

\begin{appendices}

\section{Appendix: Absolutely dominating norms}
\label{section:appendix_useful}

\textbf{In this section, $C>0$ and $D$ is an ordered Banach space with closed generating cone and with a $C$-absolutely dominating norm $\|\cdot\|$.} 

As is mentioned in \ref{obs:facts_about_absolutely_dominating}, we show that $C\ge 1$ (\ref{obs:absolutely_dominating_constant_larger_than_one}) and that for every $\epsilon>0$ there exists an equivalent $(1+\epsilon)$-absolutely norm (Theorem \ref{theorem:also_almost_1_abs_dom_norm}). Furthermore, we discuss (in \ref{obs:cN_equivalent_to_all_other_norms_with_equality} -- \ref{theorem:equivalence_normal_and_norm})  whenever there exists an equivalent norm $\|\cdot\|_1$ for which 
\begin{align}
\|x\|_1 =  \inf \{ \|a\|_1 : a\in D^+, - a\le x \le a\} \qquad (x\in D). 
\end{align}
This is done by means of the norm $\cN$ introduced in \ref{obs:norm_to_be_equal_useful}. 
Example \ref{example:absolutely_dominating_but_not_monotone_norm} illustrates that the existence of such equivalent norm may fail. 

\begin{obs}[$C$ has to be $\ge 1$]
\label{obs:absolutely_dominating_constant_larger_than_one}
Suppose that $C<1$. Choose $C'>0$ such that $C<C' <1$. 
For all $a\in D^+$ with $a\ne 0$ there exists a $b\in D^+$ with $a\le b$ and $\|b\| \le C' \|a\|$.
Let $a\in D^+$ with $\|a\|=1$. 
Iteratively one obtains a sequence $a\le a_1 \le a_2 \le \cdots$ with $\|a_1\|\le C'$ and $\|a_{n+1}\| \le C' \|a_n\|$ for all $n$. 
Then $\sum_{n\in\N} \|a_n\| <\infty$ and thus $a_n \xrightarrow{u} 0$ by Theorem \ref{theorem:equivalences_useful}, which contradicts $0<a \le a_n$. 
\end{obs}

\begin{obs}
\label{obs:norm_to_be_equal_useful}
Define $\cN: D \rightarrow [0,\infty)$ by 
\begin{align}
\cN(x) = \inf \{ \|a\|: a\in D^+, - a\le x \le a\}.
\end{align}
$\cN$ is a seminorm, and actually a norm because (see Theorem \ref{theorem:equivalences_useful})
\begin{align}
\cN(x) = 0 \iff \mbox{ there is an } a^* \in D^+ \mbox{ with } - \tfrac1n a^* \le x \le \tfrac1n a^* \quad (n\in\N). 
\end{align}
Because $\|\cdot\|$ is $C$-absolutely dominating one has $\cN \le C \|\cdot\|$. 
\end{obs}

\begin{theorem}
\label{theorem:also_almost_1_abs_dom_norm}
For all $\epsilon>0$ there exists an equivalent norm $\rho$ on $D$, for which 
\begin{align}
(1+\epsilon)^2 \rho(x) \ge \inf \{ \rho(a) : a\in D^+, - a \le x \le a\} \qquad (x\in D). 
\end{align}
\end{theorem}
\begin{proof}
Define $\rho := \epsilon \cN + \|\cdot\|$.
$\rho$ is a norm which is equivalent to $\|\cdot\|$, since $\cN \le C\|\cdot\|$. 

Let $x\in D$, $x\ne 0$. Choose $a\in D^+$ with $-a \le x\le a$ such that $\|a\| \le (1+\epsilon) \cN(x)$. Note that by definition of $\cN$ we have $\cN(a)\le \|a\|$. Whence 
\begin{align}
\frac{\rho(a)}{\rho(x)}
= \frac{  \cN(a) + \epsilon\|a\|}{  \cN(x) + \epsilon \|x\|}
\le \frac{  \cN(a) + \epsilon\|a\|}{  \cN(x) }
\le \frac{  \|a\| + \epsilon\|a\|}{  \cN(x) }
\le (1+\epsilon) \frac{\|a\|}{\cN(x)} \le (1+\epsilon)^2.
\end{align}
\
\end{proof}

\begin{obs}
\label{obs:cN_equivalent_to_all_other_norms_with_equality}
Suppose $\|\cdot\|_1$ is a norm equivalent  to $\|\cdot\|$ for which there exists a $C>0$ such that 
\begin{align}
C\|x\|_1 =  \inf \{ \|a\|_1 : a\in D^+, - a\le x \le a\} \qquad (x\in D). 
\end{align}
Then it is straightforward to show that $\cN$ is equivalent to $\|\cdot\|$. 
\end{obs}

\begin{obs}
$\cN$ is equivalent to $\|\cdot\|$ if and only if there exists a $c>0$ such that  $\cN \ge c\|\cdot\|$. 
The latter is true if and only if 
\begin{align}
-\tfrac1n a^* \le x_n \le \tfrac1n a^* \quad (n\in\N) \qquad \Longrightarrow \qquad \limn \|x_n\| =0. 
\end{align}
\end{obs}

\begin{theorem}
\begin{align}
\cN(x) = \inf \{ \cN(a): a\in D^+, -a \le x \le a\}. 
\end{align}
\end{theorem}
\begin{proof}
($\le$) Take $a\in D^+, -a \le x \le a$; we prove $\cN(x) \le \cN(a)$.
For all $b\ge a, -b \le x \le b$, whence $\cN(x) \le \|b\|$. 
Thus $\cN(x) \le \inf \{\|b\|: a\le b\} = \cN(a)$, the latter by definition of $\cN(a)$. \\
($\ge$) Take $\epsilon>0$. 
Choose $a\in D^+, -a \le x\le a$, $\|a\|\le \cN(x) + \epsilon$. 
Then $\cN(a) \le \|a\| \le \cN(x) +\epsilon$. 
\end{proof}

\begin{theorem}\cite[Theorem 2.38]{AlTo07}
For an ordered normed vector space $E$ the following are equivalent.  
\begin{enumerate}
\item The cone $E^+$ is normal. 
\item The normed space $E$ admits an equivalent monotone norm. 
\item There is a $c>0$ such that $0\le x\le y$ implies $\|x\|\le c\|y\|$.
\end{enumerate}
\end{theorem}

\begin{obs}
Suppose $D$ is an ordered Banach space with closed generating cone and suppose there exists a $c>0$ such that 
\begin{align}
\|x\| \le c \inf \{ \|a\|: a\in D^+, -a \le x \le a\} = c \cN(x). 
\end{align}
Then for $x\in D^+$, $a\in D^+$ with $x\le a$ one has $\|x\|\le c\|a\|$. 
\end{obs}

\begin{obs}
Suppose $D$ is an ordered Banach space with closed generating cone  and $\|\cdot\|$ is monotone. 
Let $x\in D$ and $a\in D^+$ be such that 
\begin{align}
-a \le x \le a. 
\end{align}
Then $0 \le x+a \le 2a$ and whence 
$
\|x\| \le \|x+a\| + \|a\| \le 3\|a\|. 
$
Thus 
\begin{align}
\|x\| \le 3 \inf \{ \|a\|: a\in D^+, -a \le x \le a\} = 3 \cN(x).
\end{align}
\end{obs}

We conclude:

\begin{theorem}
\label{theorem:equivalence_normal_and_norm}
Let $D$ be an ordered Banach space with closed generating cone.
The following are equivalent
\begin{enumerate}
\item There exist a norm $\|\cdot\|_1$ that is equivalent to $\|\cdot\|$ for which 
\begin{align}
\|x\|_1 =  \inf \{ \|a\|_1 : a\in D^+, - a\le x \le a\} \qquad (x\in D). 
\end{align}
\item There exists a $c>0$ such that $ \cN \ge c\|\cdot\|$. 
\item There exists a monotone norm that is equivalent to $\|\cdot\|$. 
\item $E^+$ is normal. 
\end{enumerate}
\end{theorem}

In the following we give an example of an ordered Banach space with closed generating cone $D$ for which none of (a) -- (d) of Theorem \ref{theorem:equivalence_normal_and_norm} holds.

\begin{example}
\label{example:absolutely_dominating_but_not_monotone_norm}
Let $D$ be $\ell^1$ with its natural norm. 
Define $T: \ell^1 \rightarrow \R^\N$ by 
\begin{align}
Tx = (x_1,x_1+ x_2, x_1+ x_2+ x_3,\dots). 
\end{align}
As $T$ is linear, $D$ is an ordered vector space under the relation $\preceq$ 
\begin{align}
x \preceq y \iff Tx \le Ty. 
\end{align}
The positive cone of $\ell^1$ is included in $D^+$, whence $D$ is directed. 
Moreover, $D^+$ is closed and so $\|\cdot\|$ is absolutely dominating. 
With $x_n = (1,-1,1,-1,\dots, \pm 1, 0,0 \dots )$ and $a=(1,0,0,\dots)$ we have $- a \le x_n \le a$ and $\|x_n\|=n$, $\|a\|=1$. 
\end{example}

\section{Appendix: Banach cover of K\"{o}the spaces}
\label{section:appendix_kothe_spaces}

In this section $(Y,\cB,\nu)$ is a complete $\sigma$-finite measure space and $M$ is the space of classes of measurable functions $Y \rightarrow \R$ as in Example \ref{example:measurable_functions_and_positive_cover}. 

\begin{lemma}
\label{lemma:special_kothe_spaces_are_banach_cover}
$\{L_{\rho_w} : w\in M^+, w>0 \mbox{ a.e.}\}$ is a Banach cover of $M$.
\end{lemma}
\begin{proof}
$L_{\rho_w}$ is complete since $f\mapsto fw$ is an isometric bijection $M_w \rightarrow L^1(\nu)$. 

Let $f\in M$. We show there exists a $w\in M^+$, $w>0$ a.e., with $f\in L_{\rho_w}$. 
By the $\sigma$-finiteness of $\nu$ there is a $u\in L^1(\nu)$, $u>0$ a.e. 
Put $w = (|f|+1)^{-1}u $. Then $w\in M^+$, $w>0$ a.e., and $f\in L_{\rho_w}$ because $\int |f| w \D \nu \le \int u \D \nu <\infty$. 

If $w, v\in M^+$, $w>0, v>0$ a.e., then $w\wedge v>0$ a.e., $L_{\rho_{w\wedge v}}$ is a subset of $L_{\rho_w}$ and $L_{\rho_v}$ and $\rho_{w\wedge v} \le \rho_{w}$ on $L_{\rho_w}$ and $\rho_{w \wedge v} \le \rho_v$ on $L_{\rho_v}$. 
\end{proof}

\begin{theorem}
The complete K\"{o}the spaces form a Banach cover of $M$. 
\end{theorem}
\begin{proof}
Let $\rho$ be a function norm and $L_\rho$ be complete. 
If $f_n \in M^+$ for $n\in\N$ and $\sumn \rho(f_n) <\infty$, then $\sumn |f_n| \in L_\rho$, so $\sumn |f_n|<\infty$ a.e. and $f_n \rightarrow 0$ a.e. 
We will use this fact below. 

By Lemma \ref{lemma:special_kothe_spaces_are_banach_cover} it suffices to prove the following. 
Let $\rho_1, \rho_2$ be function norms, $L_{\rho_1}$ and $L_{\rho_2}$ complete. 
We make a function norm $\rho$ such that $L_\rho$ is complete and $\rho \le \rho_1$, $\rho \le \rho_2$. (Then $L_{\rho_1}, L_{\rho_2} \subset L_\rho$ and we are done.)

Define $\rho : M\rightarrow [0,\infty]$ by 
\begin{align}
\rho(f) = 
\inf \{ \rho_1(g) + \rho_2(h) : g,h\in M^+, g+h \ge |f| \}. 
\end{align}
If $\rho(f) =0$, choose $g_n, h_n$ with $g_n + h_n \ge |f|$, $\rho_1(g_n) + \rho_2(h_n) \le 2^{-n}$. 
Then (by the above), $g_n \rightarrow 0$ a.e., $h_n \rightarrow 0$ a.e. Hence, $f=0$ a.e.

It follows easily that $\rho$ is a function norm. Obviously, $\rho \le \rho_1$, $\rho \le \rho_2$. 
For the completeness of $L_\rho$: 
Let $u_1,u_2,\dots \in L_\rho^+$, $\sumn \rho(u_n) <\infty$. 
Choose $g_n, h_n \in M^+$, $g_n + h_n \ge u_n$, $\rho_1(g_n) + \rho_2(h_n) < \rho(u_n) + 2^{-n}$. 
Then $\sumn \rho_1(g_n) <\infty$, so $\sumn g_n \in L_{\rho_1}$, $ \rho_1( \sumn g_n) <\infty$. 
Similarly $\rho_2(\sumn h_n) <\infty$. 
Then $\sumn u_n \in L_\rho$. 
\end{proof}
\end{appendices}

\bibliographystyle{plain}

\end{document}